\ifdefined\reviewversion
  \documentclass[lineno,pdflatex,sn-nature]{sn-jnl}
\else
  \documentclass[pdflatex,sn-nature]{sn-jnl}
\fi

\usepackage[utf8]{inputenc}
\usepackage[T1]{fontenc}
\usepackage{amsmath,amssymb,amsfonts,amsthm,bm}
\usepackage{booktabs,array}
\usepackage{graphicx}
\usepackage{microtype}
\usepackage{url}
\usepackage{xcolor}
\usepackage{siunitx}
\newcommand{\R}{\mathbb{R}}
\newcommand{\cond}{\kappa}
\newcommand{\norm}[1]{\left\lVert #1 \right\rVert}

\theoremstyle{thmstyleone}
\newtheorem{proposition}{Proposition}

\unnumbered

\begin{document}

\title[Verifier-guided discovery of mimetic operators]{Verifier-guided discovery of exact high-order mimetic operators with large language models}

\author*[1,2,4]{\fnm{J.} \spfx{de} \sur{Curt\`o}}\email{jdecurto@icai.comillas.edu}
\author[3,4]{\fnm{I.} \spfx{de} \sur{Zarz\`a}}

\affil*[1]{\orgdiv{Department of Computer Applications in Science \& Engineering}, \orgname{BARCELONA Supercomputing Center}, \orgaddress{\city{Barcelona}, \country{Spain}}}
\affil[2]{\orgdiv{Escuela Técnica Superior de Ingeniería (ICAI)}, \orgname{Universidad Pontificia Comillas}, \orgaddress{\city{Madrid}, \country{Spain}}}
\affil[3]{\orgdiv{Human centered AI, Data \& Software}, \orgname{LUXEMBOURG Institute of Science and Technology (LIST)}, \orgaddress{\city{Esch-sur-Alzette}, \country{Luxembourg}}}
\affil[4]{\orgdiv{Estudis d'Informàtica, Multimèdia i Telecomunicació}, \orgname{Universitat Oberta de Catalunya}, \orgaddress{\city{Barcelona}, \country{Spain}}}

\hypersetup{pdftitle={Verifier-guided discovery of exact high-order mimetic operators with large language models},pdfauthor={J. de Curt\`o and I. de Zarz\`a}}

\abstract{Designing a high-order structure-preserving discretization is a constrained mathematical search: conservation, a positive discrete inner product, physical spectral behavior, boundary accuracy, bandwidth, and partial differential equation (PDE) error must hold simultaneously. We test whether large language models (LLMs) can help while remaining non-authoritative. The motivating MOLE implementation of the Corbino--Castillo staggered operators satisfies a general discrete Gauss identity and conserves exactly, yet its order-six and order-eight Dirichlet blocks develop four non-real boundary-localized modes. An endpoint-supported positive-definite identity would instead force a real non-positive spectrum, so the search changes the closure and norm architecture. An LLM proposes only a typed construction program; a deterministic linear-program compiler generates coefficients; an independent verifier tests algebra, positivity, physical modes, conditioning, and manufactured PDEs; and coupled rational reconstruction provides exact certificates. Across 1,200 solver evaluations, an externally fixed verifier accepted 55.0\% of full-metric-feedback proposals and 53.3\% of illumination-archive proposals, versus 13.3\% for uniform random search. Four leading LLM-originated programs were reconstructed exactly. The strongest order-six-interior, order-four-boundary candidate lowers the prior positive-diagonal spectral-radius constant by 25.4\% and its PDE error 62.7-fold. Its certified heat-equation energy is contractive, whereas the order-six reference exhibits 6.6\% transient growth; both have the same RK4 stability limit. The LLM proposes structural hypotheses; deterministic mathematics determines validity.}


\maketitle

\section{Introduction}
High-order numerical schemes are usually presented as coefficient tables, but the difficult task is to discover a construction that satisfies several requirements at once. A boundary closure can be accurate yet destroy an energy identity; a positive norm can produce an unacceptable spectral radius; a real spectrum can hide spurious low modes; and an algebraically valid operator can still solve a simple Poisson problem poorly. Operator design is therefore a search over representations, constraints, and objectives, followed by independent verification.

Mimetic finite differences make this tension explicit because they reproduce discrete analogues of integration by parts, conservation, and compatible differential mappings \cite{castillo2003matrix,corbino2020high,lipnikov2014mimetic,bochev2006principles}. The same philosophy underlies summation-by-parts (SBP) operators \cite{strand1994summation,mattsson2013block,mattsson2014optimal}. Earlier work formalized the relevant discrete spaces, kernels, images, and quotient mappings \cite{decurto2024isomorphic}, and subsequently studied spectral and stochastic perturbation behavior \cite{decurto2024spectral}. Algorithmic SBP construction has already shown that feasibility and spectral optimization can be separated from manual coefficient design \cite{albin2016algorithmic}. The open question here is whether an LLM \cite{trinh2024alphageometry,wang2023scientific} can help choose the \emph{construction architecture} once the design space contains norm classes, asymmetric closures, row-dependent supports, and continuous objective weights.

\subsection{Problem and fixed-closure obstruction}
Let $[0,1]$ be divided into $m$ cells. The staggered scalar and vector spaces are $F_h\cong\R^{m+2}$ and $V_h\cong\R^{m+1}$, with divergence $D:V_h\rightarrow F_h$, gradient $G:F_h\rightarrow V_h$, and Laplacian $L=DG$. For homogeneous Dirichlet data, eliminating the two boundary scalar unknowns gives an exact interior block $L_I\in\R^{m\times m}$.

Let symmetric positive-definite matrices $Q$ and $P$ define scalar and vector inner products. A general matrix extended-Gauss identity has the form
\begin{equation}
QD+(PG)^T=B,
\label{e:egd}
\end{equation}
where $B$ is a discrete boundary operator. Let $E\in\R^{(m+2)\times m}$ inject interior scalar unknowns, so $L_I=E^TDGE$ and $Q_I=E^TQE$. We assume that the scalar norm respects the Dirichlet elimination, $E^TQ=Q_IE^T$, as it does for every compiled candidate.

\begin{proposition}
If $Q=Q^T\succ0$ and $P=P^T\succ0$ satisfy \eqref{e:egd}, $E^TQ=Q_IE^T$, and $E^TBGE=0$, then $L_I$ is similar to a symmetric non-positive matrix; hence its spectrum is real and non-positive.
\end{proposition}
\begin{proof}
Multiplying \eqref{e:egd} by $G$ and restricting with $E$ gives $Q_I L_I=-S+E^TBGE$, where $S=E^TG^TPGE=S^T\succeq0$; the first equality uses $E^TQ=Q_IE^T$. Under the stated support condition the last term vanishes. If $R^TR=Q_I$, then $R L_I R^{-1}=-R^{-T}SR^{-1}$, which is symmetric non-positive.
\end{proof}

The compiler below imposes the endpoint selector $B_0=-e_1e_1^T+e_{m+2}e_{m+1}^T$, for which $E^TB_0GE=0$. The reference arrays are instead the MOLE implementation of the Corbino--Castillo construction \cite{corbino2020high,corbino2024mole}, not the distinct parameterized Castillo--Grone matrix-analysis family \cite{castillo2003matrix}. They satisfy a general extended-Gauss identity with $B_{\mathrm{CC}}=QD+G^TP$, but this boundary operator extends over closure rows and $E^TB_{\mathrm{CC}}GE\ne0$. Their four non-real modes at orders six and eight (Fig.~\ref{fgr:reference_spectra}) therefore do not contradict the Corbino--Castillo identity. Rather, they show unconditionally that the fixed Dirichlet blocks admit no SPD self-adjoint similarity. Under the proposition's norm-separation hypothesis, they also rule out the stricter endpoint-supported identity with positive-definite weights. The search must change the closure, while block norms enlarge the feasible class for the new constructions.

\begin{figure}[!htbp]
\centering
\includegraphics[width=0.88\textwidth]{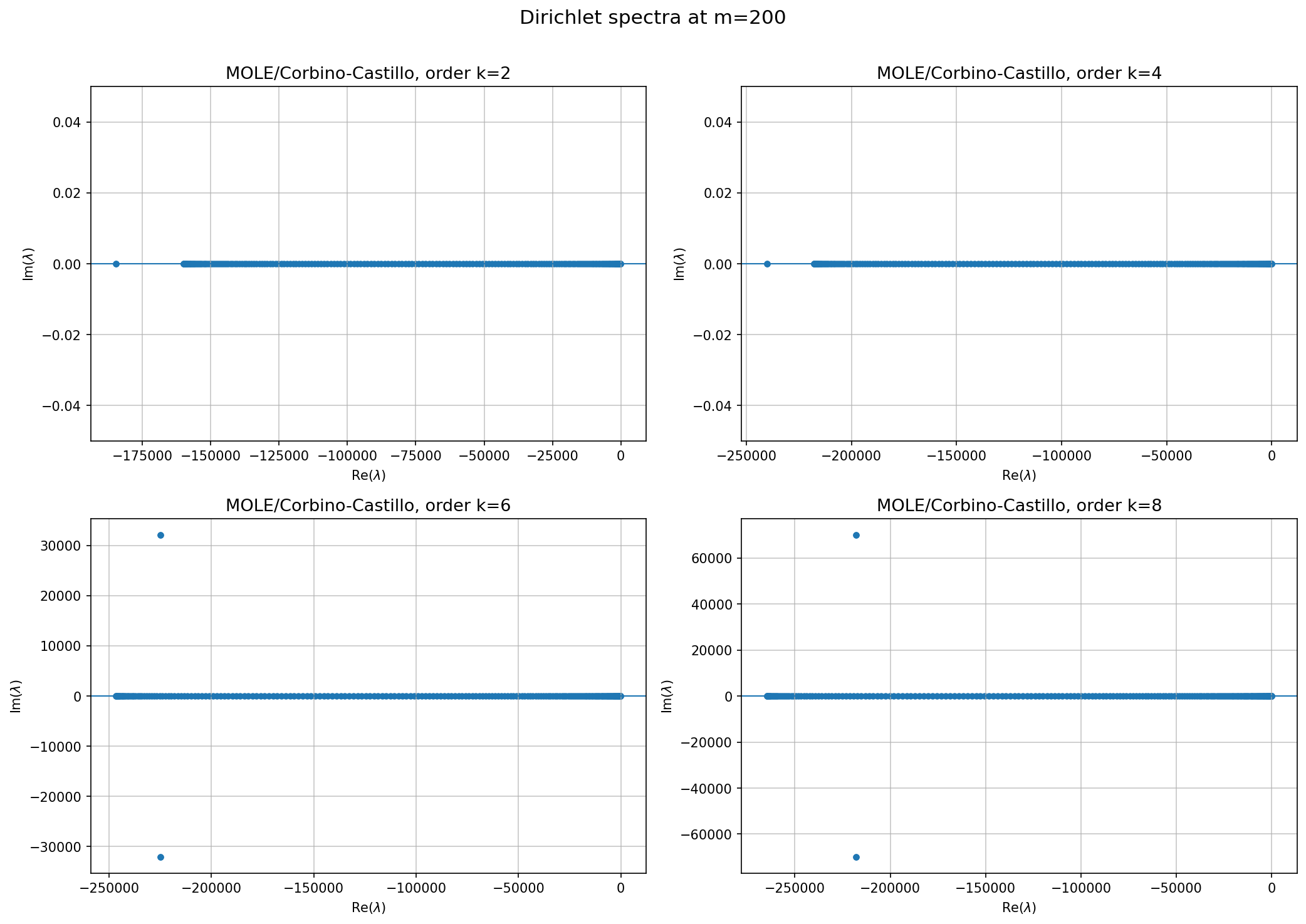}
\caption{\textbf{Reference spectra expose the Dirichlet-symmetry obstruction.} Dirichlet spectra of the MOLE/Corbino--Castillo operators at $m=200$. Orders $k=2$ and $k=4$ are real; orders $k=6$ and $k=8$ contain four non-real eigenvalues. Their general discrete Gauss identity remains valid; the non-real modes show that its closure-supported boundary term does not induce an SPD self-adjoint Dirichlet block.}
\label{fgr:reference_spectra}
\end{figure}

\subsection{Verifier-guided search}
LLMs are useful only if they do not become mathematical authorities. A model can propose a familiar stencil, invent a proof, or return coefficients that approximately satisfy a constraint. We instead ask it to emit a typed program: target order, diagonal or local block norm, reflected or independent boundary closures, row-specific support widths, and an optimization objective. A deterministic compiler solves for coefficients; a separate verifier tests algebraic identities, norm positivity, physical eigenmodes, conditioning, and PDE accuracy; and exact rational reconstruction certifies selected results (Fig.~\ref{fgr:pipeline}).

This division of labor follows evaluator-driven program search such as FunSearch \cite{romeraparedes2024funsearch} and AlphaEvolve \cite{novikov2025alphaevolve}, but the accepted object here is constrained by exact domain identities rather than an arbitrary score. A candidate is rejected if it has one spurious physical mode or one failed affine identity, regardless of its utility.

Our contribution is threefold. First, we define an open program grammar in which the model proposes a construction rather than a point in an exhaustively solved grid. Second, we combine deterministic linear compilation with an external promotion verifier and coupled exactification. Third, we compare live multi-model search with uniform random programs, a lightweight surrogate, and an expert-informed prior over 1,200 deterministic solver calls, while reporting cost, repeated-seed yield, mesh transfer, and robustness. The resulting claim is deliberately bounded: LLMs propose promising structural regions; deterministic mathematics produces and validates the operators.

\begin{figure}[!htbp]
\centering
\includegraphics[width=0.95\textwidth]{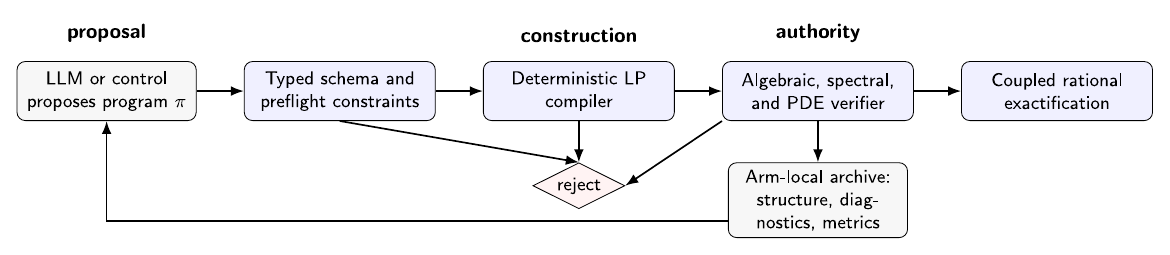}
\caption{\textbf{Verifier-guided discovery separates proposal from authority.} The LLM controls only the structural program. Coefficients, validity, and exact certification are deterministic. Rejected proposals can inform later calls but never become mathematical evidence.}
\label{fgr:pipeline}
\end{figure}

\section{Results}
\subsection{Verifier feedback increases reliable yield}
The common external verifier accepted $66/120=55.0\%$ of full-metric-feedback programs and $64/120=53.3\%$ of illumination-archive programs, compared with $32/240=13.3\%$ for uniform random search (Fig.~\ref{fgr:search}; Table~\ref{t:yield}). Both feedback packages remained significant after Holm correction ($p=0.0117$). The expert-informed prior reached 50.0\%, showing that a strong human construction prior remains competitive. The ExtraTrees surrogate reached 23.8\%, whereas structure-only and no-archive LLM conditions did not significantly exceed random search after correction.

At the common 12-call budget, the expert-informed prior was the only arm whose best-so-far area under the curve remained significant after correction. Thus the evidence supports a \emph{yield} advantage for verifier feedback, not the stronger claim that LLM search dominates expert search. With ten paired seeds, exact sign-flip randomization is discrete: the $2^{10}$ sign assignments yield raw $p$-values on a $1/1024$ grid, with two-sided floor $2/1024$. Equal adjusted values therefore do not imply equal effects; magnitude must be read from rates, intervals, and trajectories.

\begin{figure}[!htbp]
\centering
\includegraphics[width=0.88\textwidth]{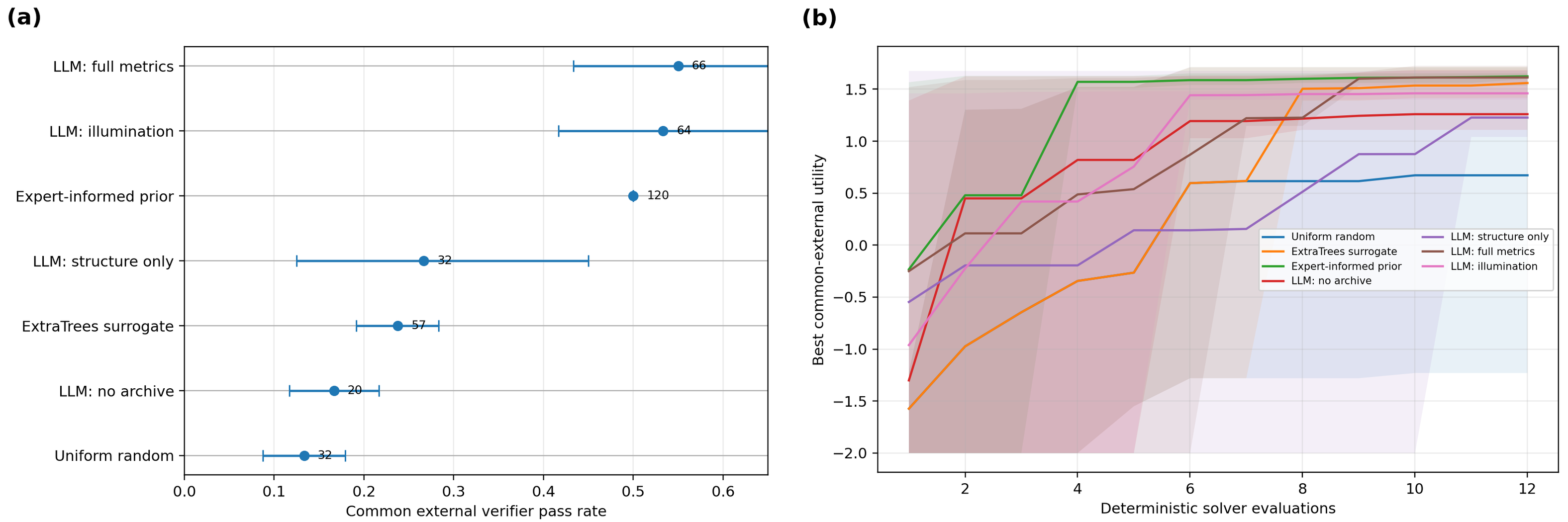}
\caption{\textbf{Feedback changes verified yield and best-so-far utility.} \textbf{(a)} Common external pass rate with seed-level bootstrap intervals; numbers indicate accepted solver calls. The expert prior uses distinct seeded sequences but yields exactly $12/24$ passes in every seed. \textbf{(b)} Best common-external utility versus deterministic solver evaluations, truncated to the common 12-call budget. Lines are seed means and bands are the 10th--90th percentiles.}
\label{fgr:search}
\end{figure}

\begin{table}[!htbp]
\caption{\textbf{Search yield under the common external verifier.} The $p$-values compare seed-level pass rates with uniform random search by exact paired sign-flip randomization tests followed by Holm correction. Passes per million tokens are computed from the complete raw generation log, including failed and repair generations.}
\label{t:yield}
\centering
\small
\begin{tabular}{@{}lrrrr@{}}
\toprule
Strategy & Calls & Passes & Rate (\%) & Holm $p$ \\
\midrule
LLM: full metrics & 120 & 66 & 55.0 & 0.0117 \\
LLM: illumination & 120 & 64 & 53.3 & 0.0117 \\
Expert-informed prior & 240 & 120 & 50.0 & 0.0117 \\
LLM: structure only & 120 & 32 & 26.7 & 0.3281 \\
ExtraTrees surrogate & 240 & 57 & 23.8 & 0.0117 \\
LLM: no archive & 120 & 20 & 16.7 & 0.4141 \\
Uniform random & 240 & 32 & 13.3 & -- \\
\bottomrule
\end{tabular}

\vspace{4pt}
\footnotesize Full metrics was the most cost-efficient LLM condition: 85.1 externally verified programs per million tokens, compared with 80.2 for illumination, 45.7 for structure only, and 37.9 for no archive.
\end{table}

A randomized diagnostic-refinement offer did not explain the feedback effect. Conditional on a preceding failed proposal, the external verifier accepted $20/42$ refined and $25/55$ fresh proposals (Fisher $p=0.840$). Among 638 LP-infeasible programs with an identified binding group, 637 (99.8\%) failed the endpoint-supported extended-Gauss compatibility equations. The dominant obstacle is therefore global structural compatibility, not a local defect that a one-step textual repair can usually fix.

A later engineering follow-up reached 98.6\% schema-valid output but only 59.1\% distinct programs; duplicate collapse and broken model--role pairing make its archive-versus-diagnostic pattern mechanism-generating rather than inferential evidence.

\subsection{Exact constructions recover weighted spectral structure}
The best program from each LLM archive condition was independently recompiled and reconstructed exactly. All four order-six-interior, order-four-boundary constructions pass the coupled affine identities, direct positive-definiteness checks, physical-mode tests, and transfer through $m=640$ (Table~\ref{t:exact}). The strongest numerical trade-off came from the structure-only GLM-5.2 proposal. Its coupled reconstruction uses denominator $2^{36}$ and changes the floating-point LP solution by at most $1.27\times10^{-9}$.

At $m=160$, this candidate has
\begin{equation}
\rho(L)h^2=6.165962,\qquad \cond_Q=1.830782,
\end{equation}
no extra low modes, no non-real eigenvalues, and mixed-frequency manufactured-PDE RMS error $8.91\times10^{-8}$. Relative to the earlier positive-diagonal $6$--$3$ candidate, it lowers the spectral-radius constant by 25.4\% and the PDE error by a factor of 62.7. The MOLE/Corbino--Castillo reference remains about 69 times more accurate on this manufactured problem, but it retains four non-real modes and has no valid certified symmetrizing-norm condition. The new operator is therefore a distinct accuracy--structure trade-off, not a uniform improvement over the reference.

\begin{table}[!htbp]
\caption{\textbf{Exact LLM-originated constructions and order-six comparison at $m=160$.} ``Exact grid'' is the dyadic denominator used in coupled reconstruction. The MOLE/Corbino--Castillo norm ratio is not reported as a certified $\kappa_Q$: its Cholesky-similarity asymmetry diagnostic is 0.481 at $m=160$. This max-norm diagnostic is distinct from the two-norm residual $r_Q$ used in Table~\ref{t:solver}.}
\label{t:exact}
\centering
\small
\begin{tabular}{@{}llrrrr@{}}
\toprule
Origin/operator & Model & $\rho h^2$ & Certified $\kappa_Q$ & PDE RMS & Exact grid \\
\midrule
No archive & MiniMax-M3 & 6.166146 & 2.139 & $2.68\times10^{-7}$ & $2^{44}$ \\
Structure only & GLM-5.2 & 6.165962 & 1.831 & $8.91\times10^{-8}$ & $2^{36}$ \\
Full metrics & MiniMax-M3 & 6.165968 & 1.959 & $1.29\times10^{-7}$ & $2^{44}$ \\
Illumination & MiniMax-M3 & 6.165961 & 1.857 & $5.11\times10^{-6}$ & $2^{36}$ \\
\midrule
MOLE/CC order 6 & -- & 6.165968 & n/a & $1.30\times10^{-9}$ & -- \\
Prior diagonal $6$--$3$ & -- & 8.266454 & 2.057 & $5.59\times10^{-6}$ & exact \\
\bottomrule
\end{tabular}
\end{table}

All four exact LLM-originated candidates remain real and free of extra low modes over $m\in\{40,64,80,96,120,160,200,320,480,640\}$. The short search-window rate can be pre-asymptotic: the leading candidate falls from an apparent 6.34 over $m=40,80,160$ to 4.67 over $m=160,320,640$, consistent with sixth-order interior and fourth-order boundary accuracy. The reference reaches a round-off floor near $m=480$, so its final apparent rate is not interpreted (Fig.~\ref{fgr:validation}).

\begin{figure}[!htbp]
\centering
\includegraphics[width=0.87\textwidth]{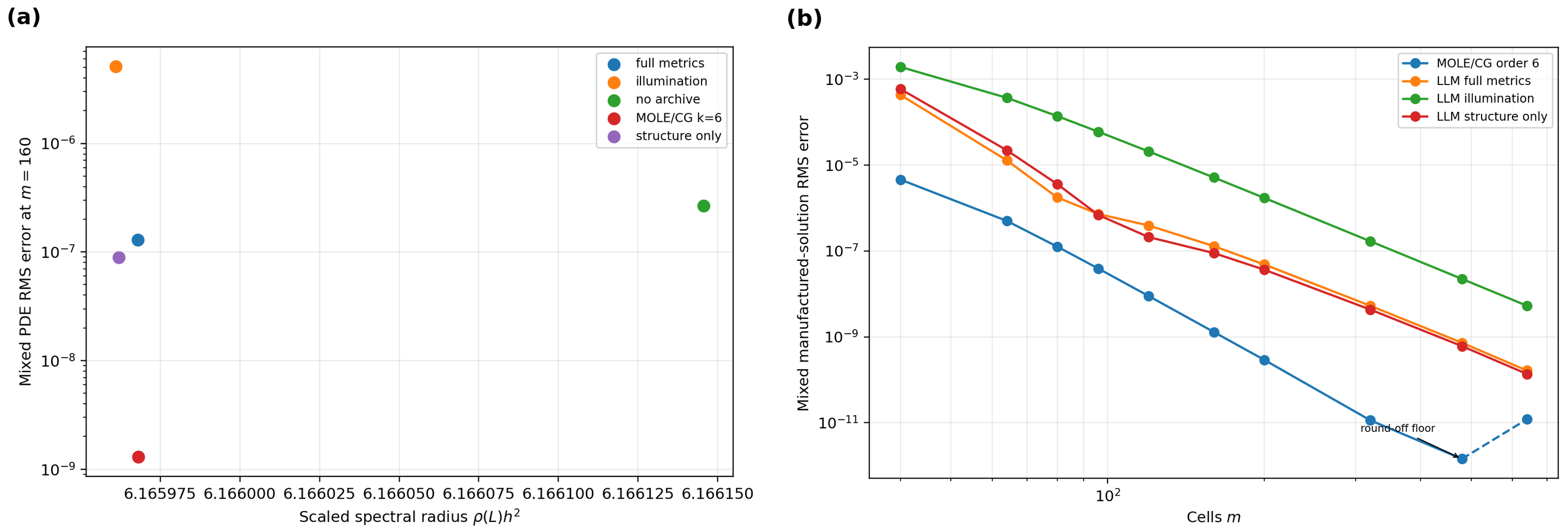}
\caption{\textbf{Exact candidates occupy a distinct accuracy--structure trade-off.} \textbf{(a)} Spectral-radius/PDE-error comparison at $m=160$. The reference is more accurate, whereas the block-norm candidates restore a real weighted spectrum at nearly the same spectral-radius constant. \textbf{(b)} Mixed-frequency PDE refinement through $m=640$. Candidate late-window rates are lower than the pre-asymptotic search estimates. The final reference segment is dashed because round-off dominates.}
\label{fgr:validation}
\end{figure}

Under boundary-supported perturbations normalized by $\norm{\Delta L}_2=\gamma\norm{L}_2$, the mean worst relative eigenvalue drift at $\gamma=0.08$ decreases from 0.01294 for the order-six reference to 0.01149 for the leading exact candidate, an approximately 11\% improvement. This is a targeted Monte Carlo diagnostic rather than a general robustness proof.

Certification also changes the semidiscrete dynamics. For $\dot u=L_Iu$ and $E_Q(u)=\tfrac12u^TQ_Iu$, the exact identity gives $\dot E_Q=-u^T(G^TPG)_{I,I}u\le0$. At $m=200$, the candidate is contractive and its normalized weighted-symmetry residual is $1.2\times10^{-16}$, versus $0.52$ for the order-six reference under its native positive quadrature. The reference's non-real spectrum precludes an SPD self-adjoint similarity, whereas the candidate admits one. The reference reaches 1.066 transient energy amplification at $t/h^2\approx0.333$, an interior maximum over $[0,0.6]$ after which it decays, with $h^2\mu_{\max}=0.328$. Both operators nevertheless have $\Delta t_{\max}/h^2=0.4517$ under RK4. Thus the recovered structure yields an energy certificate and SPD formulation, not a larger explicit step; the reference comparison remains norm-specific.

The same certificate determines which linear-algebra guarantees are available (Table~\ref{t:solver}). At $m=200$, $-Q_IL_I$ is symmetric positive definite for all four primary and three follow-up exact constructions, so conjugate gradients (CG) is mathematically applicable. None of the four MOLE/Corbino--Castillo systems meets both the weighted-symmetry and positivity conditions in its native quadrature. The table separates the error of the original discrete problem from that of the deliberately symmetrized surrogate used to interpret a nonzero residual.

\begin{table}[!htbp]
\caption{\textbf{A certified symmetrizing norm determines the available linear-algebra guarantees.} We solve the Dirichlet Poisson problem at $m=200$ with $u(x)=\sin(\pi x)+0.25\sin(7\pi x)$. The columns report the weighted-symmetry residual $r_Q$, the smallest eigenvalue $\lambda_{\min}$ of $\operatorname{sym}(-Q_IL_I)$, and the errors for the original system and the deliberately symmetrized surrogate. CG is applicable only when $r_Q<10^{-10}$ and $\lambda_{\min}>0$.}
\label{t:solver}
\centering
\footnotesize
\setlength{\tabcolsep}{3.6pt}
\renewcommand{\arraystretch}{1.08}
\begin{tabular*}{\textwidth}{@{\extracolsep{\fill}}lrrcrr@{}}
\toprule
Operator & $r_Q$ & $\lambda_{\min}$ & CG & \shortstack[c]{Original\\error} & \shortstack[c]{Symmetric-part\\error} \\
\midrule
Exact constructions (7) & $\leq1.2\times10^{-16}$ & $+0.0493$ & yes & $9.4\times10^{-9}$--$5.9\times10^{-6}$ & \multicolumn{1}{c}{same range} \\
\midrule
MOLE/CC order 2 & $7.20\times10^{-2}$ & $+0.0493$ & no & $2.41\times10^{-4}$ & $1.05\times10^{-3}$ \\
MOLE/CC order 4 & $1.49\times10^{-1}$ & $+0.0489$ & no & $3.47\times10^{-7}$ & $2.41\times10^{-2}$ \\
MOLE/CC order 6 & $5.22\times10^{-1}$ & $-26.00$ & no & $1.00\times10^{-9}$ & $2.78\times10^{-1}$ \\
MOLE/CC order 8 & $8.90\times10^{-1}$ & $-844.87$ & no & $6.15\times10^{-12}$ & $3.69\times10^{-1}$ \\
\bottomrule
\end{tabular*}

\vspace{4pt}
\footnotesize The seven exact constructions require 114--120 CG iterations, and their original-system and symmetric-part errors agree to the displayed precision. GMRES and BiCGSTAB results for the original reference systems are archived, but iteration counts are not compared across solver families.
\end{table}

For the order-six reference, CG applied deliberately to $\operatorname{sym}(-Q_IL_I)$ converges to a relative residual of $5.0\times10^{-11}$ but returns a solution in error by $0.278$. A dense direct solve of the original nonsymmetric system gives $1.00\times10^{-9}$ manufactured-solution error, and GMRES gives $1.01\times10^{-9}$. The discrepancy is therefore introduced by symmetrization, not by the reference discretization or by a particular nonsymmetric solver. This forced-CG calculation is a diagnostic, not a recommended solver or a speed comparison. Its implication is the availability of SPD guarantees---energy-norm convergence bounds, a three-term fixed-memory recurrence, and symmetric preconditioning---not an inability to solve the reference. For the order-six reference, $\lambda_{\min}=-26.0$ and $h^2\mu_{\max}=0.328$ are two views of the same quadratic-form obstruction. At $\Delta t/h^2=5$, the weighted backward-Euler system is SPD for every exact construction and for none of the four references under their native quadratures.

\section{Discussion}
The central contribution is a division of labor between generative search and mathematical authority. The LLM does not derive the accepted coefficients, prove positivity, identify physical eigenmodes, or certify exact identities. It chooses representation-level features that are expensive to enumerate: norm class, boundary symmetry, row-dependent supports, and optimization objective. The deterministic compiler and verifier then decide whether those choices define a valid operator. The observed increase from 13.3\% random yield to 55.0\% under full feedback shows that verifier information can guide models toward denser feasible regions in this grammar.

The strongest individual operator nevertheless came from the lower-yield structure-only arm, whose aggregate rate did not significantly exceed random search. This recurring tension is informative. Archive-rich feedback appears to improve the \emph{reliability} of producing acceptable constructions, whereas lower-information conditions can still generate exceptional individual objects. Best-object quality and search-policy yield are different endpoints.

Several limitations bound the interpretation. Model-role schedules were generated separately across archive conditions, so the primary feedback comparison is not fully causal. The expert-informed prior embeds substantial numerical-analysis knowledge and remains the strongest trajectory-level control. The surrogate is intentionally lightweight and should not be read as a definitive Bayesian-optimization baseline. The grammar is non-enumerable but still constrained by an empirical boundary-order preflight rule. The novelty screen is topology-scoped. No screened order-six candidate matches the loaded MOLE/Corbino--Castillo reference under canonical corner-block or fixed-length fingerprint comparison; the nearest same-order corner block remains 0.608 away and has unequal shape. This closes the Corbino--Castillo item in the partial library. Parameterized or compact higher-order Castillo--Grone extensions and generalized/block-norm staggered families remain outstanding, while nodal SBP families require an explicit topology map. We therefore report exact LLM-originated candidates, not previously unknown operator families.

The study is also one-dimensional and uniform-grid, and its block-positive cone is a sufficient inner approximation rather than the full banded semidefinite cone. A definitive next campaign should pair model-role schedules across conditions, separate archive metrics from repair text in a completed factorial design, enforce diversity after duplicate proposals, strengthen the surrogate with a feasibility model, and compare promoted operators against a complete novelty library and broader PDE suite.

Within these bounds, the mathematical result is clear. Changing the closure and norm architecture can recover exact weighted structure and a real physical spectrum at high interior order. The LLM's scientific role is specific and falsifiable: it proposes structural regions worth compiling. The mathematical claim begins only after deterministic verification and exact reconstruction. More broadly, the propose--compile--verify--exactify pattern applies whenever models can explore representational choices while validity is enforced by inexpensive deterministic or symbolic tests.

\section{Methods}
\subsection{Program representation and deterministic compilation}
A proposal specifies the target interior order, left and right boundary order, diagonal or positive local block norm, reflected or independent boundary closure, number of boundary rows, one support width per row, norm bandwidth, coefficient bounds, positivity margin, and one of four objectives: maximize norm margin, minimize coefficient $\ell_1$ norm, minimize the next unresolved moment, or a weighted hybrid. Continuous objective weights make the program space non-enumerable.

For a fixed program, we use transformed variables
\begin{equation}
A=QD,\qquad C=PG,
\end{equation}
so that the endpoint-supported extended-Gauss relation becomes the linear constraint
\begin{equation}
A+C^T=B_0,\qquad B_0=-e_1e_1^T+e_{m+2}e_{m+1}^T.
\end{equation}
Polynomial exactness is imposed as linear moment equations on shared boundary templates across several meshes. Diagonal norms use positive scalar weights. Local block norms use the LP-compatible cone
\begin{equation}
Q_b=\tau I+\sum_o \alpha_o v_ov_o^T,\qquad \tau>0,\ \alpha_o\ge0,
\end{equation}
with localized vectors $v_o$ respecting the proposed bandwidth; $P$ is parameterized analogously. This is a sufficient inner approximation of the block positive-definite cone. After solving the LP, $D=Q^{-1}A$ and $G=P^{-1}C$ are instantiated on design and holdout meshes.

\subsection{Independent verifier and exactification}
The numerical verifier first checks dimensions, finite coefficients, support, shared-template transfer, derivative moments, the two vector conservation identities, the endpoint-supported full matrix identity, and norm eigenvalue margins. For the Dirichlet block it forms a Cholesky factor $R^TR=Q_I$ and verifies symmetry of $H=RL_IR^{-1}$. Computed eigenvalues are assigned to the first five exact values $-(j\pi)^2$ by minimum-cost matching, while additional eigenvalues in the physical low-frequency window are counted separately.

Two manufactured Dirichlet problems are evaluated inside the search loop: $u(x)=\sin(\pi x)$ and $u(x)=\sin(\pi x)+0.1\sin(7\pi x)$. The primary external endpoint is fixed for every arm: compilation and algebraic gates must pass; no non-real or extra low modes are allowed; $\cond(L_I)h^2\le2$; first-mode and first-five-mode relative errors must not exceed 0.05 and 0.20; low- and mixed-frequency RMS errors must not exceed $2\times10^{-2}$ and $10^{-1}$; and the mixed-frequency rate must be at least $0.55$ times the smaller boundary order. These thresholds are proposer-independent.

Selected floating-point solutions are reconstructed in the complete coupled affine system, not coefficient by coefficient. An independent rational equation basis and reduced-row-echelon parameterization are computed; only free coordinates are placed on a common dyadic grid, pivot coordinates are solved exactly, and every original equation is rechecked. Positive-definiteness of each exact norm block is certified by its leading principal minors. The exact operator must then pass the numerical verifier on unseen meshes.

The partial novelty screen compares same-order staggered operators after canonicalizing scale, sign, reflection, and unequal closure widths. The loaded MOLE arrays are attributed to Corbino--Castillo by the pinned upstream source and compared directly. Parameterized higher-order Castillo--Grone and generalized/block-norm staggered families remain incomplete; cross-topology comparison requires an explicit degree-of-freedom map.

For the heat equation, the weighted-symmetry residual is $\lVert Q_IL_I-L_I^TQ_I\rVert_2/\lVert Q_IL_I\rVert_2$, the relative logarithmic $Q$-energy rate is the largest generalized eigenvalue of $(L_I^TQ_I+Q_IL_I,Q_I)$, and finite-time amplification is $\lVert Q_I^{1/2}e^{tL_I}Q_I^{-1/2}\rVert_2^2$. RK2 and RK4 limits are obtained by bisection on their stability polynomials over the complete spectrum. These diagnostics are downstream tests and do not enter the search objective.

Linear-solver consequences were evaluated at $m=200$ using $u(x)=\sin(\pi x)+0.25\sin(7\pi x)$. Writing $\operatorname{sym}(A)=(A+A^T)/2$, CG was declared applicable only when $-Q_IL_I$ had weighted-symmetry residual below $10^{-10}$ and $\lambda_{\min}(\operatorname{sym}(-Q_IL_I))>0$. CG on the symmetric part of a nonsymmetric weighted system was run only as a diagnostic; dense direct solves, GMRES, and BiCGSTAB were applied to the original matrix. Backward Euler used $\Delta t/h^2=5$. Iteration counts were archived but not compared across solver families because work, memory, restart, preconditioning, and numerical-library details differ.

\subsection{Search protocol and statistics}
The reference MOLE/Corbino--Castillo operators of orders $2,4,6,8$ were reproduced before search. The campaign evaluated 1,200 unique deterministic solver calls: ten seeds of 24 calls for uniform random search, an ExtraTrees mean-plus-uncertainty surrogate, and an expert-informed construction prior; and four live LLM archive conditions with ten seeds of 12 calls each. The LLM conditions were no archive, structure only, full metrics, and illumination. Three served models were retained by the capability probe: GLM-5.2, MiniMax-M3, and Cosmos3-Super-Reasoner. A fourth model was excluded by a parser that could select an echoed request schema rather than the terminal answer; this reflects the harness, not a demonstrated model limitation.

The primary endpoint is the per-call external-verifier rate; confidence intervals resample whole seed-level arms. Each strategy is compared with uniform random search by exact paired sign-flip randomization over all $2^{10}$ assignments of the observed differences, retaining their magnitudes; raw $p$-values are multiples of $1/1024$, followed by Holm correction. A secondary endpoint integrates best-so-far external utility over the first 12 calls. Token and latency totals include malformed, duplicate, rejected, and repair calls. Exact software versions and the Nebius structured-output endpoint \cite{nebius2026api} are recorded in the GitHub release.

\section*{Data availability}
The repository \url{https://github.com/drdecurto/mimetic_evolve} contains the primary live-run archive, raw model generations, deterministic solver records, exact certificates, operator arrays, derived tables, and figures.

\section*{Code availability}
The same repository contains the self-contained notebook, postprocessing and release-verification scripts, and code that recomputes provider-transport, program-diversity, missing-as-failure, completion-stratified, schedule, factorial, prediction-coverage, exact-candidate, semidiscrete-energy, and linear-solver-consequence audits.

\backmatter
\bibliography{references}

\par\smallskip\noindent\textbf{Acknowledgements.}
The authors thank the developers of MOLE and the open scientific Python ecosystem. Funding was provided by the Luxembourg Institute of Science and Technology through `ADIALab-MAST' and `LLMs4EU' (Grant Agreement No.~101198470), and by the Barcelona Supercomputing Center through `TIFON' (MIG-20232039).

\par\smallskip\noindent\textbf{Author contributions.}
J.d.C. conceived the study, developed the software and pipeline, performed the experiments and analyses, and drafted the manuscript. I.d.Z. contributed the mathematical formulation, validation, interpretation, and manuscript revision, as well as performed the experiments and analyses. Both authors approved the final manuscript.

\par\smallskip\noindent\textbf{Competing interests.}
The authors declare no competing interests. \textbf{Correspondence.} Correspondence should be addressed to J. de Curt\`o.

\end{document}